%% file: arxiv.tex
\documentclass[11pt]{article}

\usepackage[utf8]{inputenc}
\pdfoutput=1
\usepackage[margin=1in,letterpaper]{geometry}
\usepackage[round]{natbib}
\usepackage{tikz}
\makeatletter
\newlength\aftertitskip     \newlength\beforetitskip
\newlength\interauthorskip  \newlength\aftermaketitskip

\def\maketitle{\par
 \begingroup
   \def\thefootnote{\fnsymbol{footnote}}
   \def\@makefnmark{\hbox to 4pt{$^{\@thefnmark}$\hss}}
   \@maketitle \@thanks
 \endgroup
\setcounter{footnote}{0}
 \let\maketitle\relax \let\@maketitle\relax
 \gdef\@thanks{}\gdef\@author{}\gdef\@title{}\let\thanks\relax}

\def\@startauthor{\noindent \normalsize\bf}
\def\@endauthor{}
\def\@starteditor{\noindent \small {\bf Editor:~}}
\def\@endeditor{\normalsize}
\def\@maketitle{\vbox{\hsize\textwidth
 \linewidth\hsize \vskip \beforetitskip
 {\begin{center} \LARGE\@title \par \end{center}} \vskip \aftertitskip
 {\def\and{\unskip\enspace{\rm and}\enspace}%
  \def\addr{\small\it}%
  \def\email{\hfill\small\tt}%
  \def\name{\normalsize\bf}%
  \def\AND{\@endauthor\rm\hss \vskip \interauthorskip \@startauthor}
  \@startauthor \@author \@endauthor}
}}

\makeatother

\pdfoutput=1                    

\usepackage{amsmath,amsthm}
\usepackage{graphicx}
\usepackage{color}
\usepackage{amssymb,amsfonts,amsxtra,mathrsfs,bm}
\usepackage{multicol}
\usepackage{multirow}
\usepackage{paralist}
\usepackage{xspace}
\usepackage{algorithm}
\usepackage{algpseudocode}
\usepackage{hyperref}
\usepackage{bbm}
\usepackage{nicefrac}
\input{macros.tex}

\title{On a class of geodesically convex optimization problems \\ solved via Euclidean MM methods}
\author{\name Melanie Weber \email{mweber@seas.harvard.edu} \\
\addr{Harvard University}\thanks{Work partially done while at the University of Oxford.}\\
\name Suvrit Sra \email{suvrit@mit.edu}\\
  \addr{MIT, Laboratory for Information and Decision Systems}
}

\begin{document}
\maketitle

\begin{abstract}
We study geodesically convex (g-convex) problems that can be written as a difference of Euclidean convex functions. This structure arises in several optimization problems in statistics and machine learning, e.g., for matrix scaling, M-estimators for covariances, and Brascamp-Lieb inequalities. Our work offers efficient algorithms that on the one hand exploit g-convexity to ensure global optimality along with guarantees on iteration complexity. On the other hand, the split structure permits us to develop Euclidean Majorization-Minorization algorithms that help us bypass the need to compute expensive Riemannian operations such as exponential maps and parallel transport. We illustrate our results by specializing them to a few concrete optimization problems that have been previously studied in the machine learning literature. Ultimately, we hope our work helps motivate the broader search for mixed Euclidean-Riemannian optimization algorithms.
\end{abstract}

\input{1_intro}
\input{2_background_arxiv}
\input{3_mainbody}

\input{4_applications}
\input{5_conc}

\section*{Acknowledgements}
The authors thank Pierre-Antoine Absil for helpful comments on the manuscript.

\appendix
\input{appendix_v2_arxiv}

\bibliographystyle{plainnat}
\setlength{\bibsep}{4pt}
\bibliography{refs} 

\end{document}

%% file: macros.tex
\newcommand{\E}{\mathbb{E}}
\newcommand{\nlsum}{\sum\nolimits}
\newcommand{\frob}[1]{\|#1\|_{\mathrm{F}}}

\newcommand{\pd}{\mathbb{P}}
\newcommand{\reals}{\mathbb{R}}
\newcommand{\half}{\tfrac{1}{2}}
\newcommand{\nhalf}{\nicefrac{1}{2}}

\newcommand{\ip}[2]{\langle {#1},\, {#2} \rangle}
\newcommand{\norm}[1]{\| {#1} \|}

\newcommand{\gm}{\#}

\newcommand{\Mc}{\mathcal{M}}

\newcommand{\set}[1]{\left\lbrace #1 \right\rbrace}
\newcommand{\riem}{d_R}
\newcommand{\dc}{DC\xspace}
\newcommand{\sdiv}{\delta_s}
\DeclareMathOperator{\Exp}{Exp}

\def\sint{\begingroup\textstyle\int\endgroup} 

\DeclareMathOperator{\trace}{tr}
\DeclareMathOperator*{\argmin}{argmin}

\DeclareMathOperator{\grad}{grad}
\DeclareMathOperator{\ld}{ldet}

\newtheorem{theorem}{Theorem}[section]
\newtheorem{lemma}[theorem]{Lemma}

\theoremstyle{definition}
\newtheorem{defn}[theorem]{Definition}

\newtheorem{rmk}[theorem]{Remark}
\newtheorem{ass}[theorem]{Assumption}

\numberwithin{equation}{section}

%% file: 1_intro.tex
\section{Introduction}
We study optimization problems that can be written in the form
\begin{equation}
  \label{eq:1}
  \min_{x\in \Mc}\quad \phi(x) = f(x) - h(x),
\end{equation}
where $\phi: \Mc \to \reals$ is a smooth geodesically convex (\emph{g-convex}) function on a Riemannian manifold $\Mc$ (within an ambient Euclidean space), while both $f$ and $h$ are smooth Euclidean convex functions. Problems with such structure most commonly arise when $\Mc$ is the manifold of positive definite matrices, and we are estimating certain covariance or kernel matrices. Some examples relevant to machine learning where the model~\eqref{eq:1} applies include: Tyler's and related M-estimators~\citep{tyler1987distribution,wiesel2012geodesic,ollila2014regularized,sra2015conic,franks2020rigorous}; certain geometric optimization problems~\citep{sra2015conic,bacak2014convex}; metric learning~\citep{zadeh2016geometric}; robust subspace recovery~\citep{zhang2016robust}; matrix barycenters based on S-Divergence~\citep{sdiv}; matrix-square roots~\citep{sra2016matrix}; computation of Brascamp-Lieb constants~\citep{sra2018geodesically,OS2018,bennett2008brascamp}; certain Wasserstein bounds on entropy~\citep{courtade2017wasserstein}; learning Determinantal Point Processes (DPPs)~\citep{mariet2015fixed}, among others.

A powerful tool for solving g-convex problems is Riemannian optimization~\citep{absil2009optimization,boumal2020introduction,boumal2014manopt}, for which under suitable regularity conditions on the objective and its gradients both local~\citep{udriste1994convex,da1998geodesic,absil2009optimization} and global convergence rates can be attained~\citep{zhang2016first,bento2017iteration}. The corresponding algorithms typically require access to Riemannian tools such as exponential maps, geodesics, and parallel transports (in some settings retractions and vector transports may suffice). However, for g-convex problems that possess the special ``difference of (Euclidean) convex'' (\dc) structure~\eqref{eq:1}, one may wonder if simpler,  potentially more efficient methods exist.

The goal of this paper is to exploit the \dc structure of~\eqref{eq:1} via Euclidean Majorization-Minorization (MM) methods. Specifically, we use the differentiability of $f$ and $h$ to apply the Convex-Concave Procedure (CCCP)~\citep{yuille2001concave}, which is a well-known MM method that exploits the \dc structure to guarantee a monotonically decreasing sequence $\{\phi(x_k)\}_{k\ge 0}$ of objectives by successively minimizing upper bounds of $\phi(x_k)$. In general due to nonconvexity, CCCP is able to ensure only \emph{asymptotic convergence}~\citep{lanckriet2009convergence}, and at most to stationary points~\citep{le2018dc}. But our nonconvex cost function is not arbitrary: it is actually g-convex. The aim is therefore to understand how to ensure \emph{non-asymptotic} convergence guarantees for CCCP to converge to the global optimum of~\eqref{eq:1}. 

\vspace*{-4pt}
\subsection{Main contributions}
\vspace*{-4pt}
Given the above background, we summarize our main contributions below.
\begin{enumerate}
\setlength{\itemsep}{2pt}
\item We identify the \dc structure~\eqref{eq:1} within numerous Riemannian optimization problems; subsequently, we develop global non-asymptotic convergence guarantees for the CCCP algorithm (Alg.~\ref{alg:cccp}) applied to solving such problems. To our knowledge, this work presents the \textbf{first} general class of nonconvex \dc optimization problems (beyond the Polyak-Łojasiewicz class) for which global iteration complexity of CCCP could be established.
\item We illustrate the implications of using Euclidean CCCP for several applications, including  M-estimators of scatter matrices, barycenters of positive definite matrices, and computation of the Brascamp-Lieb constant (which comes up in operator scaling). 
Our theory offers a unified framework for analyzing these CCCP algorithms. Importantly, this framework provides transparent \emph{non-asymptotic convergence guarantees}, where previously such guarantees could only be obtained through a more complicated fixed-point machinery.
\end{enumerate}

While our theoretical analysis turns out be simple, in that it does not require any deep tools from Riemannian geometry, it draws attention to the important (in hindsight) realization that:
\begin{center}
\vspace*{-2pt}
  ``\emph{Many Riemannian optimization problems can be solved efficiently via a Euclidean lens.}''
\end{center}
Ultimately, we hope this realization paves the way for a broader study of mixed Riemannian-Euclidean optimization. On a more technical note, we remark that while the Euclidean view bypasses the usual Riemannian tools, and can thus potentially be computationally more efficient, the CCCP approach requires an oracle more powerful than a gradient oracle, which makes it a less general choice. Nevertheless, for several applications, we show that such an oracle is actually available. 


\vspace*{-6pt}
\subsection{Related work}
\vspace*{-6pt}
\textbf{CCCP and \dc programming.} Riemannian \dc problems have been studied recently, notably in~\citep{cruzneto,souza2015proximal,ferreira2021difference}. This line of work studies the difference of \emph{geodesically} convex functions as opposed to difference of \emph{Euclidean} convex functions. We follow a different approach that generalizes~\citep{MISO} to problems of the form~\ref{eq:1}. The methods of~\citep{souza2015proximal,ferreira2021difference} involve solving nonconvex subproblems at each iteration, whereas the Euclidean CCCP approach requires solving convex ones. 


\textbf{Riemannian optimization.} Riemannian optimization has recently seen a surge of interest in machine learning. Generalizations of classical Euclidean algorithms to the Riemannian setting have been studied for convex~\citep{zhang2016first}, nonconvex~\citep{boumal2019global}, stochastic~\citep{bonnabel2013stochastic,zhang2016first,zhang2016riemannian} and constrained problems~\citep{weber2022riemannian,weber2021projection}, among others. An introductory treatment of Riemannian optimization methods can be found in~\citep{absil2009optimization,boumal2020introduction}, whereas the works~\citep{udriste1994convex,bacak2014convex,zhang2016first} focus on the geodesically convex setting. However, there has been little work on methods that combine insights from both Euclidean an Riemannian viewpoints.

\textbf{Applications.} We discuss applications more extensively in Section~\ref{sec:appl}. 


%% file: 2_background_arxiv.tex
\vspace*{-4pt}
\section{Background and Notation}
\vspace*{-6pt}
\subsection{Riemannian manifolds}
\vspace*{-4pt}
A \emph{manifold} $\Mc$ is a locally Euclidean space that is equipped with a differential structure.  Its \emph{tangent spaces} $T_x \Mc$ consist of the tangent vectors at points $x \in \Mc$.  We focus on \emph{Riemannian manifolds}, i.e., smooth manifolds with an inner product $\ip{u}{v}_x$ defined on $T_x \Mc$ for each $x\in \Mc$.  To map between a manifold and its tangent space, we define  \emph{exponential maps} $\Exp: T_x \Mc \rightarrow \Mc$,  $y=\Exp_x(g_x) \in \Mc$, given with respect to a geodesic $\gamma : [0,1]\times \Mc \times \Mc \rightarrow \Mc$, where $\gamma(0; x, y)= x$, $\gamma(1; x, y)=y$ and $\dot{\gamma}(0; x, y)=g_x$. The \emph{inverse} exponential map 
$\Exp^{-1}: \Mc \rightarrow T_x \mathcal{M}$
defines a diffeomorphism from the neighborhood of $x \in \mathcal{M}$ onto the neighborhood of $0 \in T_x \mathcal{M}$ with $\Exp_x^{-1}(x)=0$.  The inner product structure on $T_x \Mc$ defines a norm $\norm{v}_x := \sqrt{\ip{v}{v}_x}$ for $v \in T_x \Mc$.  We define the geodesic distance of $x,y \in \Mc$ as $d(x,y)$.  Finally, we note that we limit our attention to Hadamard manifolds, i.e., complete, connected Riemannian manifolds with globally nonpositive curvature, as they present the simplest setting for discussing geodesic convexity. 

The goal of this paper is the optimization of functions $\phi: \Mc \rightarrow \reals$. If $\phi$ is differentiable, then its gradient $\grad\phi(x)$ is defined as the vector $v \in T_x \mathcal{M}$ with $D\phi(x)[v] = \ip{\grad\phi(x)}{v}_x$.
%
We say that $\phi$ is \emph{geodesically convex} (short: \emph{g-convex}), if
\begin{align} \label{eq:L2}
  \phi(y) \geq \phi(x) + \ip{\grad\phi(y)}{\Exp_x^{-1} (y)}_x 
  \qquad \forall x,y \in \mathcal{M}.
\end{align}
In the applications considered in this paper, $\Mc$ will be the manifold of positive definite matrices, i.e., 
\begin{equation}
    \Mc = \pd_d := \lbrace X \in \reals^{d \times d}: X \succ 0 \rbrace \; ,
\end{equation}
i.e., the set of all real matrices with only positive eigenvalues. We can define a Riemannian structure on $\pd_d$ with respect to the inner product
\begin{align*}
\ip{A}{B}_X = \trace \left(  X^{-1} A X^{-1} B \right) \qquad X &\in \pd_d, A,B \in T_X(\pd_d)= \mathbb{H}_d \; ,
\end{align*}
where the tangent space $\mathbb{H}_d$ is the space of symmetric matrices.  With this notion,  $\pd_d$ is a Cartan-Hadamard manifold with non-positive sectional curvature.

Throughout the paper, $\norm{\cdot}_2$ will denote the Euclidean norm.

\begin{rmk}
Observe that we did not define the usual Lipschitz continuity of Riemannian gradients, as we will not be using that in this paper. We will blend the Riemannian view with the Euclidean, and will instead require Euclidean $L$-smoothness, which for a $C^1$ function $h(\cdot)$ is defined as:
\[ \|\nabla h(x)-\nabla h(y)\|_2 \le L\|x-y\|_2. \]
\end{rmk}

\subsection{Difference of convex functions}

Our goal is to develop efficient algorithms for minimizing g-convex functions that are \emph{difference of (Euclidean) convex} (short: \dc) functions. A main motivation for us begins with problems of the form
\begin{align}
  \label{eq:2}
  \min_{X \equiv \set{X_1,\ldots,X_n} \in \pd_d} \phi(X) &:= \underbrace{-\nlsum_{i=1}^n\log\det(X_i)}_{f(X)} 
  - \underbrace{\left[-\log\det\left(\nlsum_{i=1}^n A_i^*X_iA_i\right)\right]}_{h(X)},\\
  \label{eq:3}
  \min_{X \in \pd_d} \phi(X) &:= \underbrace{- \log\det(X)}_{f(X)} 
  - \underbrace{\left[\nlsum_j -\log\det(A_j^*XA_j)\right]}_{h(X)}.
\end{align}
As shown in~\citep{sra2018geodesically}, objectives~\eqref{eq:2} are \eqref{eq:3} g-convex. Since $\log\det(X)$ is Euclidean concave, it is clear that \eqref{eq:2} and \eqref{eq:3} are \dc programs of the form~\eqref{eq:1}. We will focus on problem~\eqref{eq:3} to develop and analyze our proposed algorithm. However, very similar techniques can be used to derive an analogous approach for~\eqref{eq:2}. 



%% file: 3_mainbody.tex
\section{CCCP with global iteration complexity via g-convexity} 
We propose a Euclidean CCCP method for solving g-convex \dc problems. Our proposed method (Alg.~\ref{alg:cccp}) utilizes insights on the structure of problem~\ref{eq:1} from both Euclidean and Riemannian viewpoints.  Importantly, we exploit the g-convexity of the \dc objective to obtain a non-asymptotic iteration complexity, i.e., a non-asymptotic convergence rate to the global optimum, while exploiting Euclidean Lipschitz-smoothness to control CCCP iterates. This approach is in contrast to the standard CCCP approach that typically only guarantees asymptotic convergence. 

The analysis in this section relies on the following manifold-dependent assumption on the relation between Euclidean and geodesic distances.
\begin{ass}\label{ass}
Let $x,y \in \Mc$. 
We have $\| x - y \|_2 \leq \alpha_\Mc(d(x,y))$, where $\alpha_\Mc$ is a bounded and positive function that depends on the geometry of $\Mc$ only.
\end{ass}
\begin{rmk}
Assumption~\ref{ass} is fulfilled for important instances of problem~\ref{eq:1}. Notably, if $\Mc$ is an embedded submanifold (e.g., the unit sphere or the hyperboloid), Assumption~\ref{ass} holds with $\alpha_\Mc$ being the identity. If $\Mc = \pd_d$, which includes all applications discussed in sec.~\ref{sec:appl}, we have 
\begin{equation}
    \norm{x-y}_2^2 \leq \sqrt{2} \frac{e^{d(x,y)} - 1}{e^{d(x,y)}} \max \lbrace \norm{x}_2, \norm{y}_2 \rbrace \; .
\end{equation}
We defer the proof of this relation to the appendix.
\end{rmk}

\subsection{Algorithm}
\begin{algorithm*}
 \caption{Euclidean Majorization-Minimization scheme for Riemannian \dc} 
  \label{alg:cccp}
  \begin{algorithmic}[1]
    \State \textbf{Input:} $x_0 \in \Mc$, K
    \For{$k=0,1,\ldots,K-1$}
        \State Let $Q(x, x_k) := f(x) - h(x_k) - \ip{\nabla h(x_k)}{x-x_k}$.
        \State $x_{k+1} \gets \argmin_{x \in \Mc} Q(x,x_k)$.
    \EndFor
    \State \textbf{Output:} $x_K$
  \end{algorithmic}
\end{algorithm*}
Recall from~\eqref{eq:1} that we have $\phi(x) = f(x) - h(x)$. Since $h(x)$ is convex, $-h(x)$ is concave and we can upper bound it as
\begin{equation}
  \label{eq:7}
  -h(x) \le -h(y) - \ip{\nabla h(y)}{x-y}.
\end{equation}

To define a majorization-minimization method for Riemannian \dc problems, we build on classical CCCP methods and use gradients to linearize the concave part $-h(x)$ of the objective. CCCP utilizes the bound~\eqref{eq:7} in its update rule. In each iteration it seeks to minimize the upper bound
\begin{equation*}
  \phi(x) \le Q(x, y) := f(x)-h(y) - \ip{\nabla h(y)}{x-y}.
\end{equation*}
Since we must ensure that $x \in \Mc$, the CCCP update step (\emph{CCCP oracle}) in our case is
\begin{equation}
  \label{eq:8}
  x_{k+1} \gets \argmin_{x \in \Mc} Q(x,x_k),\qquad k=0,1,\ldots.
\end{equation}
The resulting algorithm is described schematically in Alg.~\ref{alg:cccp}.

\subsection{Convergence analysis}
The convergence of Algorithm~\ref{alg:cccp} can be established by adapting proof techniques from the convergence analysis of the MISO algorithm in the convex case~\citep[Prop. 2.7]{MISO} to g-convex Riemannian \dc problems. Our non-asymptotic convergence analysis requires us to place some regularity assumptions on the gradient $\nabla h(x)$, which we discuss below. 

We first recall the notion of first-order surrogates and recall some of their basic properties~\citep{MISO}. 
\begin{defn}[First-order surrogate functions]
\label{def:surrogate}
Let $\psi: \Mc \rightarrow \reals$. We say that $\psi$ is a \emph{first-order surrogate} of $\phi$ near $x \in \Mc$, if
\begin{enumerate}
    \item $\psi(z) \geq \phi(z)$ for all minimizers $z$ of $\psi$;
    \item the approximation error $\theta (z) := \psi(z) - \phi(z)$ is $L$-smooth, $\theta(x)=0$ and $\nabla \theta (x) =0$.
\end{enumerate}
\end{defn}
\begin{lemma}\label{lem:surrogate-prop}
Let $\psi$ be a first-order surrogate of $\phi$ near $x \in \Mc$. Let further $\theta (z) := \psi(z) - \phi(z)$ be $L$-smooth and $z' \in \Mc$ a minimizer of $\psi$. Then:
\begin{enumerate}
    \item $\vert \theta (z) \vert \leq \frac{L}{2} \norm{x-z}^2$;
    \item $\phi(z') \leq \phi(z) + \frac{L}{2} \norm{x-z}^2$.
\end{enumerate}
\end{lemma}
The proof of this lemma is straightforward. For completeness, we provide a proof in the appendix. We can now state our main convergence results:

\begin{theorem}
\label{prop:conv}
Let $d(x,x^*) \leq R$ for all $x \in \Mc$ and $\phi(x) \leq \phi(x_0)$. If the functions $Q(x,x_k)$ in Alg.~\ref{alg:cccp} are first-order surrogate functions, then
\begin{equation}
    \phi(x_k) - \phi(x^*) \leq \frac{2L\alpha_{\Mc}^2(R)}{k+2} \qquad \forall k \geq 1.
\end{equation}
\end{theorem}
To prove this theorem, we first derive a condition under which the CCCP oracle is defined via first-order surrogates:
\begin{lemma}
\label{lem:surrogate}
The function $Q(x,x_k)$ as defined in Algorithm~\ref{alg:cccp} is a first-order surrogate of $\phi$ near $x_k$, if $h$ is $L$-smooth.
\end{lemma}
\begin{proof}
  We have to show that $Q(x,x_k)$ fulfills all conditions in Definition~\ref{def:surrogate}. Note that, by construction, condition (1) is fulfilled, i.e., $Q(x, x_k) \geq \phi(x)$ for all $x$ and hence also for all minimizers. Let 
  \begin{align*}
      \theta (x) &:= Q(x,x_k) - \phi(x) \\
      &= h(x) - h(x_k) - \ip{\nabla h(x_k)}{x-x_k} \; .
  \end{align*}
  We see that this is (Euclidean) $L$-smooth whenever $h$ is (Euclidean) $L$-smooth. 
  Moreover, $\theta(x_k)=0$, and $\nabla \theta(x_k) =0$. 
\end{proof}
We can now prove Thm.~\ref{prop:conv}:
\begin{proof}
  Using the assumption that $Q(x,x_k)$ is a first-order surrogate of $\phi$ at $x_k$, Lem.~\ref{lem:surrogate} together with Lem.~\ref{lem:surrogate-prop}(ii) implies 
  \begin{align*}
      \phi(x_k) &\leq \min_{x \in \Mc} \left[ \phi(x) + \frac{L}{2} \norm{x - x_{k-1}}^2 \right] \\
      &\leq \min_{x \in \Mc} \big[ \phi(x) + \frac{L}{2} \alpha_{\Mc}^2(d(x,x_{k-1})) \big] \; ,
  \end{align*}
  where the second inequality follows from Assumption~\ref{ass}.
  We now follow Nesterov's classical proof technique~\citep{nesterov2013gradient} to see that
  \begin{align*}
    \phi(x_k) &\le \min_{s \in [0,1]}\left[\phi(\gamma(s; x_{k-1}, x^*)) + \tfrac{L s^2}{2} \alpha_\Mc^2(d(x^*,x_{k-1}))\right]\\
    &\le \min_{s \in [0,1]}\left[s\phi(x^*) + (1-s)\phi(x_{k-1}) + \tfrac{L s^2}{2} \alpha_\Mc^2(R) \right] \; ,
  \end{align*}
  where we have replaced the minimization over $x \in \Mc$ with minimization over the geodesic $\gamma(s; x_{k-1},x^*)$ and inserted the bound $d(x^*,x_{k-1})\leq R$. 
   Since $\phi(x_k)$ is a monotonically decreasing sequence, we can invoke the bounded level-set assumption ($\phi(x) \leq \phi(x_0) \; \forall x \in \Mc$) to obtain
  \begin{align}\label{eq:11}
    \phi(x_k)-\phi(x^*) &\le \min_{s \in [0,1]}\left[(1-s)(\phi(x_{k-1})-\phi(x^*)) \right. \\ 
    &\qquad \left. + \half L \alpha_{\Mc}^2(R) s^2 \right]. \nonumber
  \end{align}
  Let $\Delta_k := \phi(x_k)-\phi(x^*)$. We now have two cases:
  \begin{enumerate}
   \item If $\Delta_{k-1} \ge L\alpha_{\Mc}^2(R)$, then the optimal value of $s$ in~\eqref{eq:11} is 1, whereby we immediately have $\Delta_k \le \half L\alpha_{\Mc}^2(R)$.
  \item Otherwise, $s^* = \frac{\Delta_{k-1}}{L\alpha_{\Mc}^2(R)}$, which implies $\Delta_k \le \Delta_{k-1}\left(1- \frac{\Delta_{k-1}}{2L\alpha_{\Mc}^2(R)} \right)$ or equivalently
  \begin{align*}
      \Delta_k^{-1} &\geq \Delta_{k-1}^{-1} \left(1 - \frac{\Delta_{k-1}}{2L\alpha_{\Mc}^2(R)} \right) \\
      &\geq \Delta_{k-1}^{-1} + \frac{1}{2L\alpha_{\Mc}^2(R)} \; .
  \end{align*}
  Here, the second inequality follows from $(1-x)^{-1} \geq 1 + x \; \forall x \in (0,1)$.
  \end{enumerate}
  The claim follows from iteratively applying the two inequalities.
  %
  %
\end{proof}
\subsection{Solving the CCCP oracle}
The complexity of Alg.~\ref{alg:cccp} relies crucially on the complexity of the CCCP oracle. In the following section, we discuss several instances where the CCCP oracle has a closed form solution, resulting in a competitive algorithm. 

However, in general, a closed-form solution may not always be available. In this section, we discuss two instances of this setting. First, we investigate an \emph{inexact variant} of Alg.~\ref{alg:cccp}, where we solve the CCCP oracle only approximately. Secondly, we investigate a CCCP approach that exploits \emph{finite-sum structure} (Alg.~\ref{alg.2}), which we encounter in many problems of the form~\ref{eq:1}.

\subsubsection{Inexact CCCP oracle}
In general, we may only be able to solve the CCCP oracle approximately. Therefore, we complement our analysis of Alg.~\ref{alg:cccp} with the study of an \emph{inexact} variant. We assume that in iteration $k$, we perform an inexact CCCP update, i.e., we compute an $\epsilon$-approximate minimum
\begin{equation}
\label{eq:approx-Q}
    \tilde{Q}_k \leq \min_{x \in \Mc} Q(x,x_k) +\frac{1}{2}L\alpha_{\Mc}^2(R) s_k^2 \epsilon \; .
\end{equation}
A simple adaption of our convergence proof above gives the following non-asymptotic guarantee:
\begin{theorem}
\label{prop:conv-inexact}
Let $d(x,x^*) \leq R$ for all $x \in \Mc$, $\phi(x) \leq \phi(x_0)$ and let $Q(x,x_k)$ be first-order surrogate functions. Let  $\big(\tilde{Q}_k \big)_{k \geq 0}$ be a sequence of $\epsilon$-approximate CCCP updates in the sense of Eq.~\ref{eq:approx-Q}. 
Then
\begin{equation}
    \phi(x_k) - \phi(x^*) \leq \frac{2L \alpha_{\Mc}^2(R) (1+\epsilon)}{k+2} \qquad \forall \;  k \geq 1.
\end{equation}
\end{theorem}
The proof is a simple adaption of the proof of Thm.~\ref{prop:conv} and can be found in the appendix.

\subsubsection{Exploiting Finite-sum Structure}
\begin{algorithm*}
 \caption{Incremental Majorization-Minimization scheme for Riemannian \dc with finite-sum structure} 
  \label{alg.2}
  \begin{algorithmic}[1]
    \State \textbf{Input:} $x_0 \in \Mc$, K
    \State Set $g^0(x) := \frac{1}{m} \sum_{i=1}^m h_i(x_0) - \ip{\nabla h_i(x_0)}{x-x_0}$.
    \For{$k=1,\ldots,K$}
        \State Choose $i_k \sim [m]$ randomly. 
        \State Set $g_{i_k}^k(x) := h_{i_k}(x_k) - \ip{\nabla h_{i_k}(x_k)}{x-x_k}$ and $g_i^k \triangleq g_i^{k-1}$ for $i \neq i_k$.
        \State Set $Q(x, x_k) := f(x) - g^k(x)$.
        \State $x_{k+1} \gets \argmin_{x \in \Mc} Q(x,x_k)$.
    \EndFor
    \State \textbf{Output:} $x_K$
  \end{algorithmic}
\end{algorithm*}
In applications, we frequently encounter a version of problem~\ref{eq:1}, where $h$ has a finite-sum structure, i.e., is given by $h(x):=\frac{1}{m}\sum_{i=1}^m h_i(x)$, where the $h_i$ are $L$-smooth. Notice that in this case, computing the CCCP step requires $m$ gradient evaluations, which may be expensive, if $m$ is large. Instead of recomputing the full surrogate as in Alg.~\ref{alg:cccp}, we could make only incremental updates to the surrogate in each iteration. We outline an incremental update scheme in Alg.~\ref{alg.2},
which requires only \emph{one}, instead of $m$ gradient evaluations, significantly reducing the complexity of the CCCP oracle.

We note that several of the applications presented in section~\ref{sec:appl} have a finite-sum structure. However, in those cases, the CCCP oracle can actually be solved in closed-form, resulting in a very competitive implementation of Alg.~\ref{alg:cccp}.

For the convergence analysis we again follow closely the analysis of the MISO algorithm~\citep[Prop. 3.1]{MISO}. We show the following result:
\begin{theorem}\label{thm:finite-sum}
Let again $d(x,x^*) \leq R$ for all $x \in \Mc$ and $\phi(x) \leq \phi(x_0)$. Assume that $g_{i_k}^k$ as defined in Alg.~\ref{alg.2} is a first-order surrogate of $h_{i_k}$ near $x_{k-1}$.
Then Alg.~\ref{alg.2} converges almost surely.
\end{theorem}
We defer the proof details to the appendix.

%% file: 4_applications.tex
\section{Applications}
\label{sec:appl}
In this section we present several applications that possess the \dc structure~\eqref{eq:1}. All our examples are drawn from the manifold of positive definite matrices, since a large number of practical matrix estimation problems are known in this setting~\citep{wiesel2012geodesic,sra2015conic}, and it serves to best illustrate the practical aspects. 

Importantly, for the applications presented here, our framework provides a simple and competitive algorithm. In all cases, non-asymptotic convergence guarantees are obtained from a much simpler analysis. Our ability to solve the CCCP oracle (line 4, Alg.~\ref{alg:cccp}) in closed-form renders our approach into a practical method that is attractive for downstream applications.



%



\subsection{Matrix scaling}
In the case of diagonal positive definite matrices, g-convexity reduces to ordinary convexity after a global change of variables and one obtains convex geometric programming~\citep{boyd2007tutorial}. A canonical g-convex example is the problem of matrix scaling, for which perhaps the best known method is the classical Sinkorn algorithm~\citep{sinkhorn}, though the problem has witnessed considerable recent interest too~\citep{allen2017much,cohen2017matrix,altschuler2017near}. We comment only on the most basic version of the problem; see \citep{yuille2003concave} for more details.

We are given an $n\times n$ positive matrix $M$ for which we seek to compute diagonal scaling matrices $D$ and $E$ such that $DME$ is doubly stochastic, i.e., its rows and columns sum to $1$. Sinkhorn's algorithm is known to be obtained by applying CCCP to minimize the cost function
\begin{equation}
  \label{eq:ap:1}
  \min_{x > 0}\ \ \phi(x) := -\nlsum_j \log x_j + \nlsum_i\log\left(\nlsum_j x_j M_{ij}\right) \; .
\end{equation}
Here, $\{x_j\}$ are the diagonal elements of $E$, while the diagonal elements of $D$ are $1/\{\sum_j x_jM_{ij}\}$. Observe that $\phi(x)$ given by~\eqref{eq:ap:1} is actually g-convex,\footnote{This observation follows immediately from the  g-convexity of the BL problem~\eqref{eq:2}, of which problem~\eqref{eq:ap:1} is known to be a special case.} and thus this problem is indeed of the form~\eqref{eq:1}. Further, it can be verified that the $h(x)$ part of~\eqref{eq:ap:1} satisfies the $L$-smoothness assumption, since the logarithm is L-smooth on a domain with a positive lower bound.

\subsection{Tyler's M-estimator}
Estimating the shape of a covariance matrix for high-dimensional data is an important problem in statistics. One important class of covariance estimators, based on elliptically contoured distributions, is Tyler's M-estimator\\ \citep{ollila2014regularized}. There are several important asymptotic properties of this estimator, and it has been extensively studied; for additional details and discussion we refer the reader to the papers~\citep{franks2020rigorous,sra2015conic,wiesel2012geodesic,wiesel2015structured,ollila2014regularized,zhang2016robust,tyler1987distribution}. The best known algorithms for computing Tyler's M-estimator arise from carefully constructed fixed-point iterations. The convergence analysis of those fixed-point iterations utilize the Hilbert projective metrics, in a manner analogous to Birkhoff's use of the Hilbert projective metric for the convergence analysis of problems closely related to matrix scaling~\citep{birkhoff1957extensions}. Following our discussion above, Algorithm~\ref{alg:cccp} delivers a transparent method for obtaining Tyler's estimator by solving~\eqref{eq:tyler}, at least in the cases where g-convexity applies; see also \citep{sra2015conic} for additional discussion.

The resulting optimization problem involves obtaining a scatter matrix by maximizing a likelihood of the form
\begin{equation}
\label{eq:tyler}
\mathcal{L}(X, A) := -\frac{n}{2}\log\det(X) + \nlsum_i \log f(a_i^T X^{-1}a_i),
\end{equation}
where $f$ is a so-called ``distance generating function''. The likelihood~\eqref{eq:tyler} generalizes the usual multivariate Gaussian to the much larger class of Elliptically contoured distributions. Assuming that $\log f$ is concave and monotonic, it is easily seen that~\eqref{eq:tyler} can be equivalently written as a g-convex minimization problem (after reversing signs) of the form~\eqref{eq:1}. Empirically, the much faster run times obtained via CCCP (which yields a fixed-point iteration for solving~\eqref{eq:tyler}) has been explicitly highlighted in~\citep{sra2015conic,hosseini2016inference}.

\subsection{Matrix square root and barycenter of PD matrices}\label{sec:4.3}
\begin{figure*}[ht]
    \centering
    \hfill
    \includegraphics[width=0.4\textwidth]{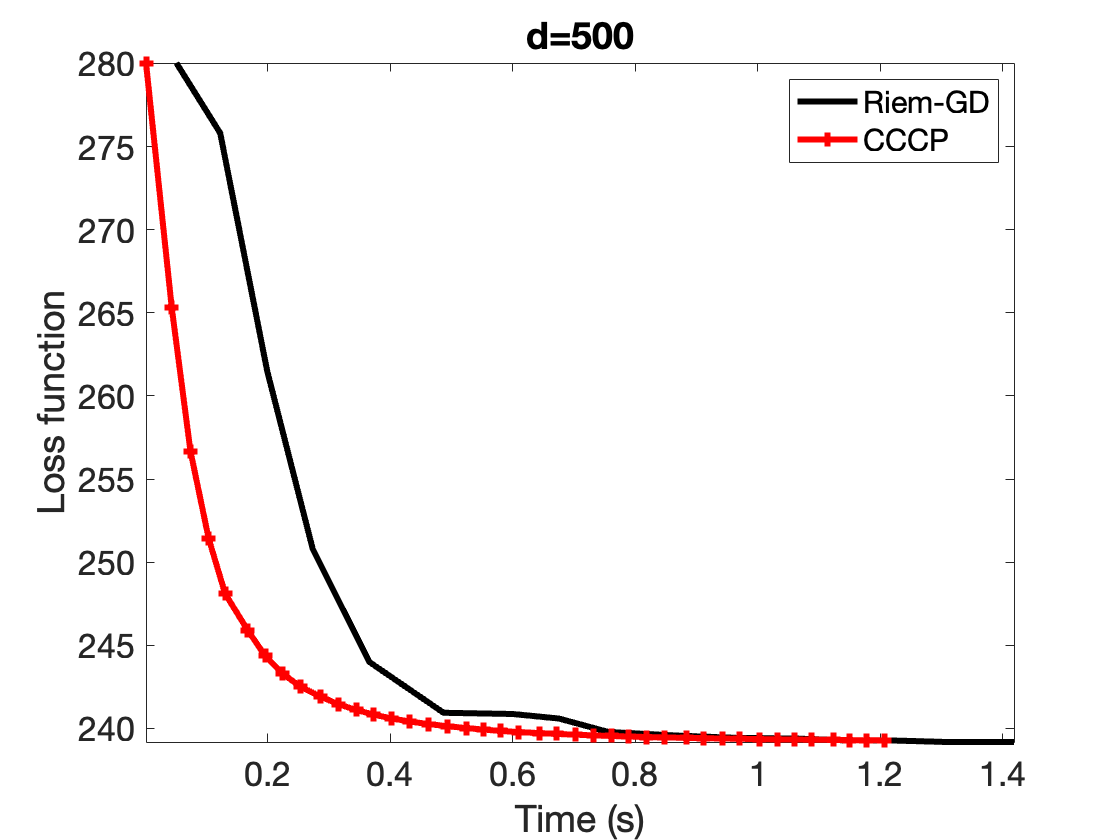}
    \hfill
    \includegraphics[width=0.4\textwidth]{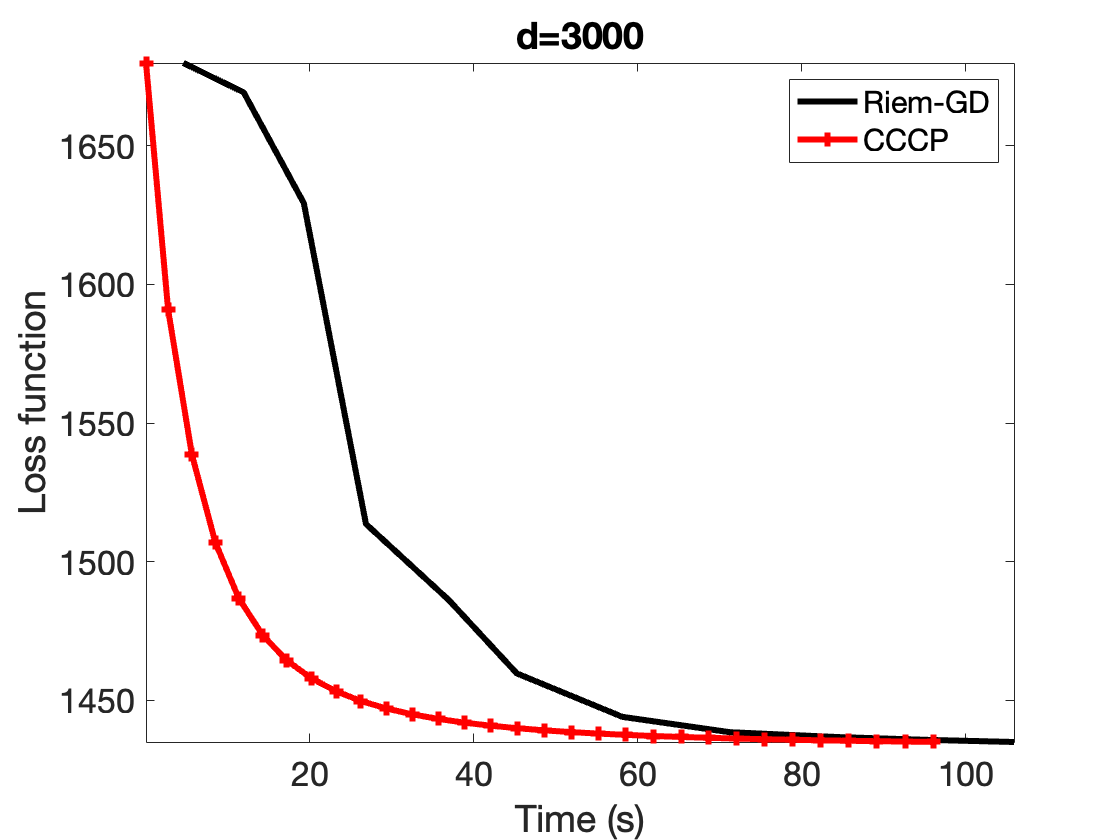}
    \hfill
    \caption{Performance of CCCP approach in comparison with Riemannian Gradient Descent for computing matrix square roots with  inputs of dimension $d$.}
    \label{fig:sqrt}
\end{figure*}
The S-Divergence~\citep{sdiv} between two positive definite matrices $X, Y$ is defined as
\begin{equation}
\label{eq:sdiv}
    \sdiv(X,Y) := \log\det\left(\tfrac{(X+Y)}{2}\right)-\half\log\det(X)-\half\log\det(Y).
\end{equation}
\paragraph{Matrix square root.} Suppose $M$ is a positive definite matrix. In~\citep{jain2017global} the authors proposed a gradient descent based method to compute the square root of $M$. A faster algorithm was obtained in~\citep{sra2016matrix} who proposed the following iteration 
\begin{equation*}
    X \gets (X+I)^{-1} + (X+M)^{-1},
\end{equation*}
which was obtained as a certain fixed-point iteration to compute the \emph{barycenter} 
\begin{equation}
\label{eq:mtxroot}
    \min_{X \succ 0}\quad \sdiv^2(X,I)+\sdiv^2(X,M).
\end{equation}

\paragraph{Barycenter.} More generally, in~\citep{sdiv} the barycenter version of \eqref{eq:mtxroot} is studied. Here, given $n$ positive definite matrices $A_1,\ldots,A_n$, one seeks to solve
\begin{equation}
\label{eq:smean}
    \min_{X \succ 0}\quad \sum_{i=1}^n w_i \sdiv^2(X,A_i).
\end{equation}
Using the defintion~\eqref{eq:sdiv} of $\sdiv$, it is immediate that \eqref{eq:smean} is a difference of Euclidean convex functions; its g-convexity is more involved but follows from~\citep{sdiv}. By applying our CCCP Algorithm, one immediately recovers a proof of convergence for the fixed-point iteration proposed in~\citep{sdiv} for solving~\eqref{eq:smean}.

\subsection{Brascamp-Lieb Constant}\label{sec:4.4}
\begin{figure*}[ht]
    \centering
    \hfill
    \includegraphics[width=0.4\textwidth]{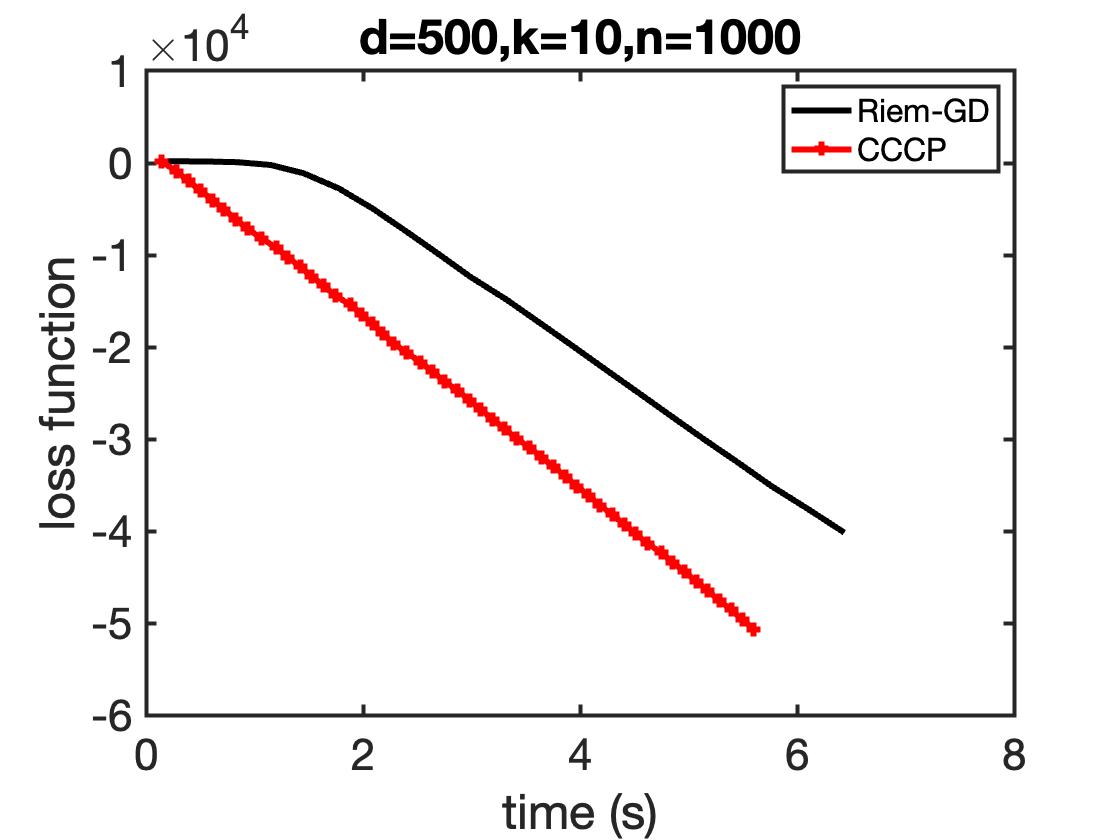}
    \hfill
    \includegraphics[width=0.4\textwidth]{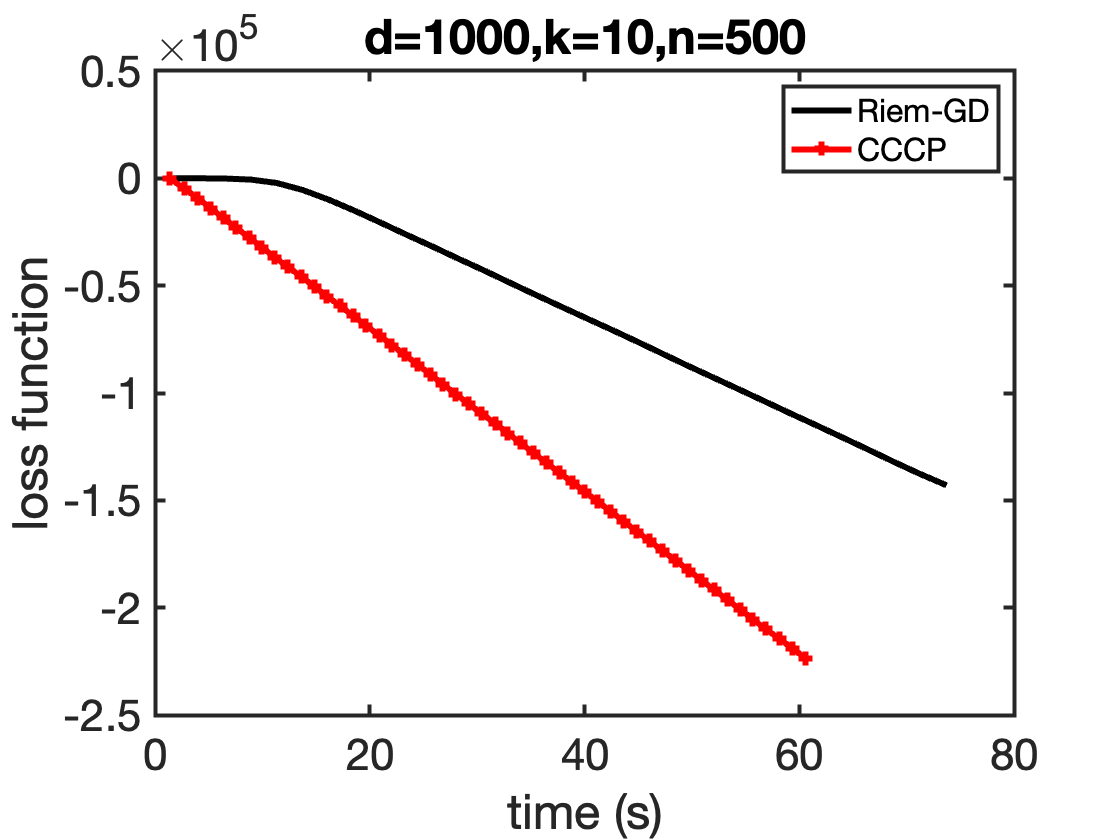}
    \hfill
    \caption{Performance of CCCP approach in comparison with Riemannian Gradient Descent for computing Brascamp-Lieb constants for $n$ input matrices of size $d \times k$.}
    \label{fig:BL}
\end{figure*}
Now we come to what is perhaps the most interesting application of the problem structure~\eqref{eq:1}. Indeed, as previously, this application is the one that motivated us to develop the method studied in this paper. Specifically, we study Brascamp-Lieb (short: BL) inequalities that form a central class of inequalities in functional analysis and probability theory, offering a great generalization to the basic H\"older inequality, and being intimately related with entropy inequalities too. As a special instance of the Operator Scaling problem~\citep{garg2016algorithmic}, they relate to a range of problems in various areas of mathematics and theoretical computer science~\citep{bennett2008brascamp}.


The computation of BL constants can be formulated as an optimization task on $\pd_d$:
\begin{align}\label{eq:BL}
  F(X) = -\log\det(X) + \nlsum_iw_i\log\det(\Phi_i(X)),
\end{align}
where $\Phi_i(X)=A_i^* X A_i$, and $w_i \ge 0$ with $w^T1=1$.
Notably, the objective is g-convex~\citep{sra2018geodesically}, which allows for applying Algorithm~\ref{alg:cccp} with global convergence guarantees.
Since $\log\det(\Phi(X))$ is a concave function of $X$, it can be upper bounded as
\begin{equation}
  \label{eq:8}
  \log\det(\Phi_i(X)) \le \log\det(\Phi(Z)) + \trace(Z^{-1}\Phi_i(X)-Z).
\end{equation}
Using~\eqref{eq:8} we thus have the following upper bound
\begin{align*}
  F(X) &\le -\log\det(X) + \log\det(Z) + \nlsum_iw_i\trace(Z^{-1}\Phi_i(X)) - d \\
  &=: g(X,Z).
\end{align*}
The CCCP update step is
\begin{equation}
  \label{eq:9}
  X_{k+1} \gets \argmin_{X >0} g(X,X_k), 
\end{equation}
which results in iteration of the map
\begin{equation}
  \label{eq:10}
  X_{k+1} = \left[\nlsum_i w_i A_i(A_i^*X_kA_i)^{-1}A_i^*\right]^{-1}.
\end{equation}
Our analysis delivers 
non-asymptotic guarantees for computing BL constants, a result that was obtained by analyzing a more involved operator Sinkhorn iteration in~\citep{garg2016algorithmic}, as well as, more recently in~\citep{weber2022computing} with more involved tools from Finslerian geometry.

\subsection{Experiments}
To demonstrate the efficiency of our proposed approach, we complement our discussion with experimental results for two of the applications discussed above. We show that CCCP performs competitively against Riemannian Gradient Descent (RGD) for the problem of computing matrix square roots (Fig.~\ref{fig:sqrt}) and for computing Brascamp-Lieb constants (Fig.~\ref{fig:BL}). In both experiments we compare against the Manopt~\citep{boumal2014manopt} version of RGD for two different sets of hyperparameters. We do note, however, that this advantage in running time is more pronounced for larger problems, as expected.




%% file: 5_conc.tex
\section{Conclusion}
We consider geodesically convex optimization problems that admit a Euclidean difference of convex (\dc) decomposition. We analyzed the global iteration complexity of Euclidean CCCP applied to solving such problems, where geodesic convexity played an important role to help bound function suboptimality, while Euclidean smoothness (of one of the \dc components) helped control the progress of CCCP. While simple, this work captures a sufficiently valuable class of nonconvex optimization problems for which CCCP can be shown to converge globally. We illustrate our ideas on several important applications where such a \dc structure arises, and for which CCCP either delivers a new convergent algorithm, or helps us explain the convergence of an existing algorithm.

An important question in this context is whether there exist an
efficiently computable \dc representation for any
geodesically convex cost function? 
Since $\Mc$ is a manifold, it is an open set. Hence, nonconvex nonsmooth functions that satisfy bounded-variation  admit a DC representation; moreover, in case $\phi \in C^2$ (i.e., twice continuously differentiable), there always is a DC representation, regardless of g-convexity~\citep{konno1997dc}.
The key challenge is whether one can efficiently find such a representation. This problem seems to be of considerable difficulty. In Appendix~1.1 we give an example of how the well-known Riemannian distance function $\riem(X,Y)=\frob{\log(X^{-1/2}YX^{-1/2})}$ on the positive definite matrices admits such a \dc representation, albeit one that seems quite intricate as it involves integrating over infinitely many functions.

We hope that our work spurs not only an investigation of the fundamental question raised above, but of better algorithms and complexity analysis for CCCP and other related procedures when applied to the class of g-convex functions studied in this paper. We believe that it should be possible to drop the dependence on the gradient Lipschitzness $L$ in the CCCP method studied in this work, but expect that a completely different approach will be needed to analyze the method. Finally, in the same vein, it will be valuable to extend our study to non-differentiable g-convex functions that enjoy a Euclidean \dc representation. We leave investigation of these important problems to the future.

%% file: appendix_v2_arxiv.tex

\section{\dc representation of Riemannian distance on PD matrices}
\label{app:riem}
We will need the following useful integral representation of the squared logarithm:
\begin{lemma}
  Let $x > 0$. Then, the following representation holds
  \begin{equation}
    \label{ap:5}
    (\log x)^2 = \int_0^\infty [\log(1+tx) + \log(t+x) - 2\log x - 2\log(1+t)]\frac{dt}{t}.
  \end{equation}
\end{lemma}

We are now ready to state our result on the \dc representation of Riemannian distance.
\begin{theorem}
  Let $X, Y \in \pd_n$ and let $\riem(X,Y) = \frob{\log (X^{-\nhalf}YX^{-\nhalf})}$ be the Riemannian distance between them. Then, $\riem^2(X,Y)$ is g-convex jointly in $(X,Y)$ and it admits a \dc representation
  \begin{equation}
    \label{ap:2}
    \riem^2(X,Y) = \int_0^\infty [f_t(X,Y) - h_t(X,Y)]d\mu(t),
  \end{equation}
  where $f_t$ and $h_t$ are convex, and $\mu$ is a suitable measure.
\end{theorem}
\begin{proof}
  It is well-known that $\riem$ is jointly g-convex---see e.g., \citep[Ch.6]{bhatia2009positive} for a proof. Consequently, $\riem^2$ is also g-convex. In deriving our proof of the \dc representation of $\riem^2$, we will also obtain an alternative (and to our knowledge, a new) proof of this joint g-convexity as a byproduct.

  Begin with observing that $\riem^2(X,Y) = \frob{\log (X^{-\nhalf}YX^{-\nhalf})} = \sum_i (\log \lambda_i(X^{-1}Y))^2$. For brevity, we write $\lambda_i \equiv \lambda_i(X^{-1}Y)$ and $\ld \equiv \log\det$; then, using the integral~\eqref{ap:5} we have
  \begin{align*}
    \riem^2(X,Y) &= \sum_i\sint_0^\infty \left[\log(1+t\lambda_i) + \log(t+\lambda_i) - 2\log\lambda_i- 2\log(1+t) \right] \frac{dt}{t}\\
    &= \sint_0^\infty\left[\ld(I+t X^{-1}Y) + \ld(t I + X^{-1}Y) - 2\ld(X^{-1}Y)-2n\log(1+t)\right]\frac{dt}{t}\\
    &= \sint_0^\infty\left[\ld(X+tY)+\ld(tX+Y) -4\ld(X)-2\ld(Y)-2n\log(1+t)\right]\frac{dt}{t}\\
    &= \sint_0^\infty \left[-4\ld(X)-2\ld(Y)-2n\log(1+t)  - (-\ld(X+tY)-\ld(tX+Y)) \right]\frac{dt}{t}\\
    &= \sint_0^\infty \left[f_t(X,Y) - h_t(X,Y) \right]\frac{dt}{t},
  \end{align*}
  where $f_t(X,Y) = -4\ld(X)-2\ld(Y)-2n\log(1+t)$ and $h_t(X,Y) = -\ld(X+tY)-\ld(tX+Y)$. Convexity of both $f_t$ and $h_t$ is immediate from the well-known convexity of $-\log\det(X)$ on $X \succ 0$.
\end{proof}

\section{Proof Details}
\subsection{Relation between Euclidean and Riemannian metrics on $\pd_d$}
\begin{lemma}
  Let $x,y \in \pd_d$. Then the Euclidean and Riemannian distance relate as
  \begin{equation*}
      \norm{x-y}_2^2 \leq \sqrt{2} \frac{e^{d(x,y)} - 1}{e^{d(x,y)}} \max \lbrace \norm{x}_2, \norm{y}_2 \rbrace \; .
  \end{equation*}
\end{lemma}
\begin{proof}
  Recall that the Thompson metric $\delta_T$ and the Riemannian distance $d$ for positive definite matrices are given by
  \begin{align*}
      \delta_T(x,y) &:= \norm{\log \big(x^{-\half} y x^{-\half} \big)} \\
      d(x,y) &:= \norm{\log \big(x^{-\half} y x^{-\half} \big)}_F \; ,
  \end{align*}
  where $\norm{\cdot}$ denotes the operator norm and $\norm{\cdot}_F$ the Frobenius norm. It is well-known that $\norm{x} \leq \norm{x}_F$ for $x \in \pd_d$. This implies $\delta_T(x,y) \leq d(x,y)$. The claim follows from a relation between the Euclidean distance and the Thompson metric, established by~\citet{snyder}:
  \begin{equation*}
      \norm{x-y}_2^2 \leq \sqrt{2} \frac{e^{\delta_T(x,y)} - 1}{e^{\delta_T(x,y)}} \max \lbrace \norm{x}_2, \norm{y}_2 \rbrace \; .
  \end{equation*}
\end{proof}


\subsection{Properties of surrogate functions}
We sketch a proof for Lem.~3.4:
\begin{lemma}
Let $\psi$ be a first-order surrogate of $\phi$ near $x \in \Mc$. Let further $\theta (z) := \psi(z) - \phi(z)$ be $L$-smooth and $z' \in \Mc$ a minimizer of $\psi$. Then:
\begin{enumerate}
    \item $\vert \theta (z) \vert \leq \frac{L}{2} \norm{x-z}^2$;
    \item $\phi(z') \leq \phi(z) + \frac{L}{2} \norm{x-z}^2$.
\end{enumerate}
\end{lemma}
\begin{proof}
For (1) recall a classical inequality, which follows from the $L$-smoothness of the surrogate function:
\begin{equation*}
    \vert \theta(z) - \theta(x) - \ip{\nabla \theta(x)}{x-z} \vert \leq \frac{L}{2} \norm{x-z}_2^2 \; .
\end{equation*}
The claim follows from $\theta(x)=0$ and $\nabla \theta(x)=0$. 

For (2), note that we have by construction
\begin{equation*}
    \phi(z') \leq \psi(z') \leq \psi(z) = \phi(z) - \theta(z) \; .
\end{equation*}
Inserting (1) directly gives the claim.
\end{proof}
\color{black}

\subsection{Inexact CCCP oracle}
For completeness, we give a proof of Thm.~3.7:
\begin{theorem}
Let $d(x,x^*) \leq R$ for all $x \in \Mc$, $\phi(x) \leq \phi(x_0)$ and let $Q(x,x_k)$ be first-order surrogate functions. Let  $\big(\tilde{Q}_k \big)_{k \geq 0}$ be a sequence of $\epsilon$-approximate CCCP updates in the sense of Eq.~3.6. 
Then
\begin{equation}
    \phi(x_k) - \phi(x^*) \leq \frac{2L \alpha_{\Mc}^2(R) (1+\epsilon)}{k+2} \qquad \forall \;  k \geq 1.
\end{equation}
\end{theorem}
\begin{proof}
  Replacing exact with inexact CCCP updates, we have
  \begin{equation*}
      \phi(x_k) \leq \min_{x \in \Mc} \left[ \phi(x) + \frac{L}{2} \|x-x_{k-1}\|^2  + \half L \alpha_\Mc^2(R) s^2 \epsilon \right] \; .
  \end{equation*}
  Following the steps of the proof of Thm.~3.5 to Eq.(3.5), we get
  \begin{align*}
      \phi(x_k)-\phi(x^*) &\le \min_{s \in [0,1]}\left[(1-s)(\phi(x_{k-1})-\phi(x^*)) + \half L \alpha_\Mc^2(R) s^2 (1+\epsilon) \right] \; .
  \end{align*}
  The claim follows from an analysis the step-sizes analogously to the proof of Thm.~3.5.
\end{proof}

\subsection{Exploiting Finite-sum Structure}
We give a proof of Thm.~3.8:
\begin{theorem}
Let again $d(x,x^*) \leq R$ for all $x \in \Mc$ and $\phi(x) \leq \phi(x_0)$. Assume that $g_{i_k}^k$ as defined in Alg.~2 is a first-order surrogate of $h_{i_k}$ near $x_{k-1}$.
Then Alg.~2 converges almost surely.
\end{theorem}
\begin{proof}
  As outlined in Alg.~2, we use the following majorization to construct the CCCP oracle:
  \begin{align*}
      g^k(x) &:= \frac{1}{m} \sum_{i=1}^m g_i^k(x) \\
      g_i^k(x) &= \begin{cases}
        h_{i_k} (x_k) - \ip{\nabla h_{i_k}(x_k)}{x-x_k}, &{\rm if} \; i=i_k \\
        g_i^{k-1}, &{\rm if} \; i \neq i_k
      \end{cases} \; .
  \end{align*}
  By construction, this gives
  \begin{equation}\label{eq:surr-update}
      g^k(x) = g^{k-1}(x) + \frac{g_{i_k}^k(x) - g_{i_k}^{k-1}(x)}{m} \; .
  \end{equation}
  Observe that
  \begin{align*}
      g^k(x_k) &\overset{(1)}{\leq} g^k(x_{k-1})\\
      &\overset{(2)}{=} g^{k-1}(x_{k-1}) + \frac{g_{i_k}^k(x_{k-1}) - g_{i_k}^{k-1}(x_{k-1})}{m} \\ 
      &\overset{(3)}{=} g^{k-1}(x_{k-1}) + \frac{h_{i_k}(x_{k-1})-g_{i_k}^k(x_{k-1})}{m} \\
      &\overset{(4)}{\leq} g^{k-1}(x_{k-1}) \; .
  \end{align*}
  Here, (1) follows from $x_k$ being the argmin determined in the CCCP step and (2)
  from Eq.~\ref{eq:surr-update}. By assumption, $g_{i_k}^k$ is a first-order surrogate of $h_{i_k}$ near $x_{k-1}$, which implies by Def.~3.3(2) that $g_{i_k}^k(x_{k-1})=h_{i_k}(x_{k-1})$ and therefore (3). (4) follows from $g_{i_k}^{k-1}$ being a majorization of $h_{i_k}$. With this, $\lbrace \big( g^k(x_k)\big) \rbrace_{k \geq 0}$ is monotonically decreasing. Due to the level-set assumption this ensures that the sequence converges almost surely. 
  
  Taking expectations in the chain of inequalities, we get monotone convergence of $\lbrace \E \big[ g^k(x_k) \big] \rbrace_{k \geq 0}$. For the analysis of the approximation error $\big(g^k(x_k)-h(x_k) \big)$, note that
  \begin{align*}
      \E \big[ \sum_{k=0}^\infty g_{i_k}^k (x_k) - h_{i_{k+1}}(x_k)  \big] 
      &\overset{(5)}{=} \sum_{k=0}^\infty \E \big[ g_{i_{k+1}}^k (x_k) - h_{i_{k+1}}(x_k) \big] \\
      &\overset{(6)}{=} \sum_{k=0}^\infty \E \Big[ \E \big[ g_{i_{k+1}}^k (x_k) - h_{i_{k+1}} (x_k) \vert \mathcal{F}_k \big] \Big] \\
      &= \sum_{k=0}^\infty \E \big[ g^k(x_k) - h(x_k) \big] \\
      &\overset{(5)}{=} \E \big[ \sum_{k=0}^\infty g^k(x_k) - h(x_k) \big] < \infty \; .
  \end{align*}
  Here, (5) follows from the Beppo Levi lemma; in (6), we have rewritten the previous equality with respect to the sigma-field $\mathcal{F}_k$ generated by the $x_k$. With this, we have that $\lbrace \big( g^k(x_k) - h(x_k)\big) \rbrace \rightarrow 0$ almost surely. Now, following the proof of Thm.~3.5, we conclude that the sequence of objective values generated by Alg.~2 convergences to the optimum almost surely.
\end{proof}
%